\documentclass{amsart}
\calclayout
\usepackage{amssymb,amsmath,amsfonts,epsfig,latexsym,tikz}
\usepackage[alphabetic]{amsrefs}
\usepackage{tikz-cd}
\definecolor{mystery}{rgb}{0,0,0.65}
\usepackage[
  colorlinks=true,
  linkcolor=mystery,
  citecolor=mystery,
  urlcolor=mystery
]{hyperref}
\usepackage{enumerate}
\usepackage{mathtools}
\usepackage{verbatim}
\usepackage{cleveref}
\usepackage[shortlabels]{enumitem}
\usepackage{subcaption}

\usetikzlibrary{positioning}
\usetikzlibrary{matrix}
\usetikzlibrary{decorations}
\usetikzlibrary{decorations.pathreplacing, decorations.pathmorphing, angles,quotes}
 
\newtheorem{theorem}{Theorem}[section]

\newtheorem{lemma}[theorem]{Lemma}

\newtheorem{conjecture}[theorem]{Conjecture}

\theoremstyle{definition}
\newtheorem{remark}[theorem]{Remark}
\newtheorem{example}[theorem]{Example}

\newcommand{\W}[1]{\ensuremath{W^{#1}}}

\newcommand{\floor}[1]{\ensuremath{\left\lfloor #1 \right\rfloor}}

\newcommand{\abs}[1]{\ensuremath{\left\lvert #1 \right\rvert}}
\begin{document}

\title{Matroid flat counts can have many peaks}

\makeatletter
\def\author@andify{%
  \nxandlist{\unskip, }{}{\unskip, }
}
\makeatother

\author{Alexander Divoux}
\address{Program in Applied and Computational Mathematics, Princeton University}
\email{adivoux@princeton.edu}
\author{Matt Larson}
\address{Princeton University and the Institute for Advanced Study}
\email{mattlarson@princeton.edu}
\author{Chayim Lowen}
\address{Department of Mathematics, Princeton University}
\email{chayiml@princeton.edu}
\author{Shouda Wang}
\address{Program in Applied and Computational Mathematics, Princeton University}
\email{shoudawang@princeton.edu}
 
    
\date{\today}
\begin{abstract}
We disprove Rota's conjecture that the counts of flats in a matroid according to rank form a unimodal sequence. Furthermore, we show that this sequence can have arbitrarily many peaks. The construction starts by finding a generalized theta graph for which log-concavity fails severely. By taking direct sums, we break log-concavity in many places. We then use Whittle's $q$-lift construction to produce a matroid whose flat counts have many peaks. 
\end{abstract}
 
\maketitle

\section{Introduction}

For a matroid $M$ and a natural number $i$, let $W_i$ be the number of flats of rank $i$. These are known as the Whitney numbers of the second kind of $M$. 
In his article for the proceedings of the 1970 International Congress of Mathematicians, Gian-Carlo Rota stated his conjecture that these numbers form a unimodal sequence for any matroid \cite{RotaICM}.

\begin{conjecture}\label[conjecture]{conj:rota}
For any matroid $M$ of rank $r$, the sequence $(W_0, W_1, \dotsc, W_r)$ is unimodal. That is, there is some index $k$ such that $W_0 \le W_1 \le \dotsb \le W_{k-1} \le W_k$ and $W_k \ge W_{k+1} \ge \dotsb \geq W_{r-1} \geq W_r$. 
\end{conjecture}

In \cite{Mason}, Mason conjectured three strengthenings of Conjecture~\ref{conj:rota}, the weakest of which is that Whitney numbers are log-concave. These conjectures have attracted significant attention over the years. For example, Mason's log-concavity conjecture is Problem 25(c) in Stanley's list of positivity problems in algebraic combinatorics \cite{StanleyPositivity}. 

In making Conjecture~\ref{conj:rota}, Rota was motivated by the known cases of Boolean lattices, braid matroids \cites{HarperBraid, LiebBraid}, and perfect matroid designs \cite{YoungDesign}; see \cite[{pg.\ 69}]{RotaMatching}. The log-concavity conjecture was proved for certain supersolvable matroids \cite{Damiani} and for Dowling geometries \cite{StonesiferMatroid,Benoumhani}. See \cite{Aigner} for a survey. 

The special case of Mason's log-concavity conjecture at $k=2$, that $W_2^2 \ge W_1 W_3$, has attracted particular attention. This conjecture, which is known as the points-lines-planes conjecture, was proved for graphic matroids in \cite{Stonesifer} and for matroids where each flat of rank $2$ has size at most four in \cite{Seymour}. In fact, what was proved there is the strongest of Mason's log-concavity conjectures for these matroids: that $W_2^2 \ge \frac{3}{2} \frac{W_1 - 1}{W_1 - 2} W_1 W_3$. The points-lines-planes conjecture was subsequently studied by Kung \cite{KungPLP} and Dukes \cite{Dukes}. 

The most significant progress towards Conjecture~\ref{conj:rota} was the proof of the Dowling--Wilson top-heavy conjecture \cites{DW1,DW2}: if $M$ is a matroid of rank $r$ then $W_i \le W_{r-i}$ for $i \le r/2$ and $W_0 \le W_1 \le \dotsb \le W_{\lfloor r/2 \rfloor}$. This was proved for realizable matroids in \cite{HW} and then for all matroids in \cite{BHMPW20b}. A simpler proof was given in \cite{LefschetzModule}. 

As is explained in \cite{KungGeometric}, Conjecture~\ref{conj:rota} was motivated by the Alexandrov--Fenchel inequality for mixed volumes of convex bodies. In recent years, Rota's intuition has been partially validated, as analogues of the Alexandrov--Fenchel inequality have been used to affirmatively resolve several log-concavity conjectures about matroids \cites{Huh2012,HK12,AHK18,BH,ALOV}. See \cite{HuhICM1} for a discussion of the analogy between the Alexandrov--Fenchel inequality and these results. 

We give a counterexample to Conjecture~\ref{conj:rota}.
Better yet, we show that unimodality can fail very badly, in the following sense. In a sequence $(a_0, a_1, \dotsc, a_s)$ of real numbers, a \emph{peak} is an index $i$ such that $a_i > a_{i-1}$ and $a_i = a_{i + 1} = \dotsb = a_{i + t} > a_{i + t + 1}$ for some nonnegative integer $t$, where we interpret $a_{-1}$ and $a_{s+1}$ as $-\infty$. 
A sequence is unimodal if and only if it has a unique peak. 

\begin{theorem}\label[theorem]{thm:main}
For each positive integer $m$, there is a matroid whose sequence of Whitney numbers of the second kind has exactly $m$ peaks. 
\end{theorem}

This refutes Conjecture~\ref{conj:rota}.
Recall that a sequence ($a_0, a_1, \dots, a_s)$ of nonnegative real numbers is \emph{log-concave} if $a_i^2 \geq a_{i-1} a_{i+1}$ for every $0 < i < s$. As log-concave sequences with no internal zeroes are unimodal, \Cref{thm:main} also gives a counterexample to Mason's log-concavity conjecture. The smallest example of a matroid that our method gives whose Whitney numbers are not unimodal has a ground set of size $4957$; see Example~\ref{ex:nonunimodal}. The smallest example of a matroid that we can construct whose Whitney numbers are not log-concave has a ground set of size $78$; see Example~\ref{ex:nonlogconcave}. The unique failure of log-concavity in the latter example is that $W_{73}^2 < W_{72} W_{74}$. Our techniques seem incapable of producing a counterexample to the points-lines-planes conjecture. 

\medskip

We prove Theorem~\ref{thm:main} by a construction that proceeds in two steps. First, we exhibit examples where log-concavity of the Whitney numbers fails very badly. For this, we use \emph{generalized theta graphs}. Let $G_t$ be the graphic matroid of the generalized theta graph consisting of two endpoint vertices and four internally disjoint paths joining them, where three of the paths have length $t$ and the remaining path consists of a single edge. See Figure~\ref{fig:logconc} for a depiction of $G_{28}$. 
This is a matroid of rank $3t-2$ on a ground set of size $3t + 1$. We show that the log-concavity of the Whitney numbers fails badly for this family of matroids: the quantity $W_{3t-5} W_{3t-3}/W_{3t-4}^2$ grows linearly in $t$. Moreover, by taking direct sums of these matroids, we can produce examples where log-concavity fails at many indices. 

\def\Len{10}      
\def\dotr{1.5pt} 
\def\bigr{3pt}    
 
\tikzset{
  edge/.style={line width=.8pt, line cap=round, line join=round},
  vert/.style={fill=black},
  term/.style={fill=black},
}
 
\newcommand{\longpath}[2]{
  \pgfmathsetmacro{\Rr}{(\Len*\Len)/(8*(#1)) + (#1)/2}
  \pgfmathsetmacro{\Cy}{(#1) - \Rr}
  \pgfmathsetmacro{\angv}{atan2(\Rr-(#1), -\Len/2)}
  \pgfmathsetmacro{\angw}{atan2(\Rr-(#1),  \Len/2)}
  \pgfmathtruncatemacro{\pathlength}{28}
  \pgfmathtruncatemacro{\lastvertex}{\pathlength - 1}
  \draw[edge] (0,0)
    \foreach \i in {1,...,\pathlength} {
      -- ({\Len/2 + \Rr*cos(\angv + (\i/\pathlength)*(\angw-\angv))},{(#2)*(\Cy + \Rr*sin(\angv + (\i/\pathlength)*(\angw-\angv)))})
    };
  \foreach \i in {1,...,\lastvertex} {
    \fill[vert] ({\Len/2 + \Rr*cos(\angv + (\i/\pathlength)*(\angw-\angv))},{(#2)*(\Cy + \Rr*sin(\angv + (\i/\pathlength)*(\angw-\angv)))}) circle (\dotr);
  }
}

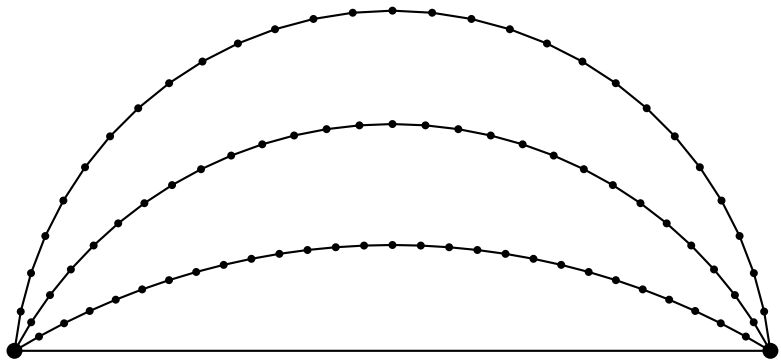
\begin{figure}
    \begin{tikzpicture}
      \draw[edge] (0,0) -- (\Len,0);
     
      \longpath{3.0}{ 1}   
      \longpath{1.4}{ 1}   
      \longpath{4.5}{ 1}   
     
      \fill[term] (0,0)    circle (\bigr);
      \fill[term] (\Len,0) circle (\bigr);
      \node at (0,4.7) {}; 
    \end{tikzpicture}
    \caption{The generalized theta graph whose matroid is $G_{28}$.}
    \label{fig:logconc}
\end{figure}

To produce many peaks and prove Theorem~\ref{thm:main}, we use the following elementary observation: a sequence $(a_0, a_1, a_2, \dotsc, a_m)$ is log-concave with no internal zeroes if and only if, for every $c > 0$, the sequence $(a_0, c a_1, c^2 a_2, \dotsc, c^m a_m)$ is unimodal. To exploit this, we use an operation introduced by Whittle \cite{Whittle89} called the $q$-lift. This operation turns a matroid with Whitney numbers $(W_0, W_1, W_2, W_3, \dotsc)$ realized over $\mathbb{F}_q$ into a matroid with Whitney numbers $(W_0, q W_1 + W_0, q^2 W_2 + W_1, q^3 W_3 + W_2, \dotsc)$. Choosing $q$ judiciously, and using that direct sums of $G_t$ with itself are realizable over any field, we produce examples of matroids whose Whitney numbers have arbitrarily many peaks. 

\medskip

The construction used to prove Theorem~\ref{thm:main} has similarities to two prominent constructions that were used to give counterexamples to other conjectures about matroids. In \cite{Sokal}, Sokal used generalized theta graphs (with a varying number of paths) to show that the roots of chromatic polynomials of graphs are dense in the complex plane. The coefficients of chromatic polynomials of graphs (and more generally characteristic polynomials of matroids) are known as Whitney numbers of the \emph{first} kind, and they can be viewed as counts of flats weighted by the M\"{o}bius function. However, it was shown in \cite{AHK18} that Whitney numbers of the first kind are log-concave. 

The dual of the matroid $G_t$ is a non-simple matroid whose simplification is the uniform matroid $U_{3,4}$. In \cite{MeroniWelsh}, a non-simple matroid whose simplification is uniform was used to refute the Merino--Welsh conjecture, which predicted an inequality for the Tutte polynomial of all matroids. 

\subsection*{Organization of this paper}

The rest of this paper is organized as follows. Section \ref{sec:lcc} is devoted to the proofs of Lemmas \ref{lm:gap} and \ref{lm:asymptotics}, which exhibit failures of log-concavity for the Whitney numbers of the second kind. \Cref{ex:nonlogconcave} provides an explicit matroid for which Mason's conjecture fails, along with the relevant flat counts. In Section \ref{sec:unimodality}, we prove \Cref{thm:main}. \Cref{ex:nonunimodal} gives a matroid whose Whitney numbers of the second kind form a non-unimodal sequence.

\subsection*{Acknowledgments}

The authors thank Paul Seymour and Luis Ferroni for helpful discussions. 
Part of this work was conducted at the 2026 \textit{Combinatorics at the Confluence} conference, and the authors thank the organizers and Carnegie Mellon University for their hospitality.
This work was conducted while the second author was at the Institute for Advanced Study, where he is supported by the Charles Simonyi Endowment and the Oswald Veblen Fund. 

In producing this work, ChatGPT played an important role in finding some initial examples which the authors were able to generalize. These aspects of our constructions have been detailed in the preprints \cite{LarsonCounterexamples} and \cite{DivouxLowenWang}, which are not intended for publication. Beyond these initial examples, AI tools did not meaningfully assist in the results or writing of this paper.

\section{Failures of log-concavity}\label{sec:lcc}

We will make repeated use of Knuth's asymptotic notation, as set out in \cite{Knuth}. In particular, for functions $f,g$ on $\mathbb N$ valued in the nonnegative reals, we write $f = \Theta(g)$ to mean that there exist constants $C, D > 0$ and an integer $N$ such that $Cg(n) \leq f(n) \leq Dg(n)$ for all $n \geq N$.

In the following lemma and in the subsequent text, we write $\W{\eta}(M)$ for the Whitney number $W_{r(M)-\eta}(M)$ of a matroid $M$, i.e. the number of flats of $M$ of \emph{corank} $\eta$.
We also write $E(M)$ for the ground set of $M$ and $r_M$ for its rank function. We write $r(M)$ for
$r_M(E(M))$, the {rank} of $M$ as a matroid. We refer the reader to \cite{Oxley} for other standard conventions in matroid theory.
Recall the matroid $G_t$, which was defined as the graphic matroid of the graph consisting of two endpoints and four internally disjoint paths between them, three of length $t$ and one of length $1$.

\begin{lemma}\label[lemma]{lm:gap}
    For a fixed integer $\eta \ge 0$, as $t$ grows to infinity, the Whitney numbers of $G_t$ are given asymptotically by
    \[
        \W{\eta}(G_t) = \begin{cases}
            1 & \text{if $\eta = 0$,}\\
            \Theta(t^{\eta + 2}) & \text{if $\eta \in \{1,2\}$,}\\
            \Theta(t^{\eta + 3}) & \text{if $\eta \geq 3$.}
        \end{cases}
    \]
\end{lemma}
\begin{proof}
    In the dual $G_t^*$, the four series classes of $G_t$ become parallel classes $P_0, P_1, P_2, P_3$, with $\abs{P_0}=1$ and $\abs{P_1}=\abs{P_2}=\abs{P_3}=t$. Simplifying these parallel classes yields the uniform matroid $U_{3,4}$, so each circuit of $G_t^*$ is either a two-element subset of one of $P_1, P_2, P_3$, or has size four and consists of one element from each parallel class. For every flat $F$ of $G_t$ of rank $0 \le j \le r(G_t)$, the corresponding cyclic set $S = E(G_t)\setminus F$ of $G_t^*$ satisfies
    \[
        \abs{S}-r_{G_t^*}(S) = r(G_t) - r_{G_t}(F) = r(G_t) - j
    \]
    by the dual rank formula. We will refer to the quantity $\abs{S} - r_{G_t^*}(S)$ as the \emph{nullity} of $S$ (in $G_t^*$). By the equation above, computing $\W{\eta}(G_t)$ amounts to counting cyclic sets $S$ of $G_t^*$ of nullity $\eta$. The equality $\W{0}(G_t) = 1$ is clear.

    A cyclic set $S$ of nullity $1$ is simply a circuit. By the above classification of the circuits of $G_t^*$, the number of cyclic sets of nullity $1$ is
    \[
        \W{1} = 3\binom t2 + t^3 = \Theta(t^3).
    \]

    Let $S$ be a cyclic set of nullity 2. Then $S$ is the union of two distinct circuits of $G_t^*$, so
    \[
        \abs{S} = r_{G_t^*}(S) + 2 \leq r(G_t^*) + 2 = 5.
    \]
    If $S$ contains $P_0$, there are at most four elements of $S$ in $P_1\cup P_2\cup P_3$. If $S$ does not contain $P_0$, then it is the union of two size 2 circuits, and therefore has at most four elements.
    In either case, $S$ is determined by a subset of $P_1 \cup P_2 \cup P_3$ of size 3 or 4. Hence
    \[
        \W{2} \leq 2 \left({3t \choose 4} + {3t \choose 3}\right) = O(t^4).
    \]
    On the other hand, the union of a pair of elements of $P_1$ with another pair from $P_2$ is a cyclic set of nullity 2. So,
    \[
        \W{2} \geq {t \choose 2}^2 = \Omega(t^4).
    \]
    We deduce that $\W{2} = \Theta(t^4)$.

    Now we prove the claim for $\eta \geq 3$. On the one hand, a set $S$ of nullity $\eta$ satisfies 
    \[
        \abs{S} = r_{G_t^*}(S) + \eta \leq r(G_t^*) + \eta = 3 + \eta.
    \]
    It follows that
    \[
        \W{\eta} \leq \sum_{j=0}^{\eta + 3} {3t + 1 \choose j} = O(t^{\eta + 3}).
    \]
    On the other hand, any set consisting of two elements from $P_1$, two elements from $P_2$, and $\eta - 1$ elements from $P_3$ is cyclic (since $\eta - 1 \geq 2$), has size $\eta + 3$, has rank 3, and thus has nullity $\eta$. Hence
    \[
        \W{\eta} \geq {t \choose 2} {t \choose 2} {t \choose \eta - 1} = \Omega(t^{\eta + 3}).
    \]
    We therefore have $\W{\eta} = \Theta(t^{\eta + 3})$ for $\eta \ge 3$, as desired.
\end{proof}

\begin{example}\label[example]{ex:nonlogconcave}
Consider the matroid of the generalized theta graph with four strands of lengths $1, 25, 26, 26$. This is a matroid of rank $75$ on a ground set of size $78$. We have 
$$(W_{72}, W_{73}, W_{74}) = (49120475, 933425,17850).$$
The Whitney numbers of this matroid are not log-concave. 
\end{example}

\begin{remark}
The argument above works just as well for generalized theta graphs with more strands (of lengths $1, t, \dotsc, t$) or strands of various other lengths. It can also be embedded within a larger matroid, in the following sense. Choose a cocircuit $D$ of size at least $4$ in a matroid $M$. Let $M_t$ be the matroid obtained by cothickening three of the elements of this cocircuit. That is, $M_t$ is the matroid obtained by taking the dual of $M$, replacing three of the elements of $D$ with parallel classes of size $t$, and then taking the dual again. We then have $W^1(M_t) = \Theta(t^3)$, $W^2(M_t) = \Theta(t^4)$, and $W^3(M_t) = \Theta(t^6)$, so log-concavity will fail for $t$ sufficiently large. 
\end{remark}

We write $M^{\oplus m}$ for the direct sum of $m$ copies of a matroid $M$. Taking \Cref{lm:gap} as a starting point, we now compute asymptotics for the Whitney numbers of $G_t^{\oplus m}$.

\begin{lemma}\label[lemma]{lm:asymptotics}
    For fixed integers $\eta \ge 0$ and $m \ge 1$, the Whitney numbers of $G_t^{\oplus m}$ as $t$ grows to infinity are given asymptotically by
        \[
        \W{\eta}(G_t^{\oplus m}) = \begin{cases}
            \Theta(t^{3\eta}) & \text{if $\eta \leq m$,}\\
            \Theta(t^{\floor{3(\eta + m)/2}}) & \text{if $m \leq \eta \leq 3m$,}\\
            \Theta(t^{\eta + 3m}) & \text{if $\eta \geq 3m$.}
        \end{cases}
        \]
\end{lemma}
\begin{proof}
    It is easily verified that this agrees with \Cref{lm:gap} when $m = 1$. The flats of the direct sum of two matroids consist of all pairwise unions of a flat from the first matroid and a flat from the second, and the rank of such a flat is the sum of the ranks of its two parts. Therefore, the sequence $(W^{\eta}(G_t^{\oplus m}))_{\eta}$ is the $m$-fold convolution of the sequence of $(W^{\eta}(G_t))_{\eta}$ with itself. In other words,
    \[
        \W{\eta}(G_t^{\oplus m}) = \sum_{x_1 + \dots + x_m = \eta}~\prod_{i=1}^m \W{x_i}(G_t),
    \]
    with the implicit constraint $x_1,\dots, x_m \ge 0$. Combining this with \Cref{lm:gap}, we conclude that $\W{\eta}(G_t^{\oplus m}) = \Theta(t^{f(\eta, m)})$ for the unique integer-valued function $f$ satisfying
    \[
        f(\eta, m) = \sup_{x_1+\dots + x_m = \eta}~\sum_{i=1}^m f(x_i,1),
    \]
    where the values $f(\eta, 1)$ are given by
    \[
        f(\eta, 1) = \begin{cases}
            0 & \text{if $\eta = 0$,}\\
            \eta + 2 & \text{if $\eta \in \{1,2\}$,}\\
            \eta + 3 & \text{if $\eta \geq 3$.}
        \end{cases}
    \]
    It will be convenient to rewrite the latter equalities as follows.
    Write $f(\eta, 1) = \eta + c(\eta)$. Then $c(0) = 0$, $c(1) = c(2) = 2$, and $c(\eta) = 3$ for $\eta \ge 3$. So by the formula for $f(\eta, m)$,
    \[
        f(\eta, m) = \eta + \sup_{x_1+\dots + x_m = \eta}~\sum_{i=1}^m c(x_i).
    \]
    To compute $f(\eta, m)$, we must 
    choose $x_1, \dots, x_m$ to maximize the sum of the $c(x_i)$'s. To this end, we split into three cases according to whether $\eta \le m$, $m \le \eta \le 3m$, or $\eta \ge 3m$. In each case, we prove an upper bound and then show that equality is attained by an explicit choice of $x_i$'s.

    First suppose that $\eta \le m$. Since $c(x) \le 2x$ for every $x \ge 0$, we have $\sum_{i=1}^m c(x_i) \le 2\eta$, so $f(\eta, m) \le 3\eta$. On the other hand, setting $x_1 = \cdots = x_\eta = 1$ and $x_{\eta + 1} = \cdots = x_m = 0$ achieves equality. So in fact $f(\eta, m) = 3\eta$.

    Next suppose that $m \le \eta \le 3m$. Since $c(x) \le (x+3)/2$ for every $x \ge 0$, we get the inequality $\sum_{i=1}^m c(x_i) \le (\eta + 3m)/2$. Since $f(\eta, m)$ is an integer, this gives
    \[
        f(\eta, m) \le \eta + \left\lfloor\frac{\eta+3m}{2}\right\rfloor = \left\lfloor \frac{3(\eta + m)}{2}\right\rfloor.
    \]
    The lower bound for this case is slightly more involved. Write $\eta - m = 2a + b$ with $b\in\{0,1\}$. We set $x_1 = \cdots = x_a = 3$, $x_{a+1} = \cdots = x_{a + b} = 2$, and $x_{a+b+1} = \cdots = x_m = 1$. (Note that $\eta \le 3m$ implies $a + b \le m$.) For these values, using $c(3) = 3$ and  $c(1) = c(2) = 2$ gives $\sum_{i=1}^m c(x_i) = 3a + 2b + 2(m - a - b) = a + 2m$. Furthermore, 
    \[
        \left\lfloor\frac{\eta+3m}{2}\right\rfloor = \left\lfloor\frac{2a + b + 4m}{2}\right\rfloor = a + 2m + \left\lfloor\frac b2\right\rfloor = a + 2m.
    \]
    We conclude that $f(\eta, m) = \eta +  \lfloor(\eta+3m)/2\rfloor = \lfloor 3(\eta + m)/2\rfloor$.

    Finally, suppose that $\eta \ge 3m$. Since $c(x) \le 3$ for every $x \ge 0$, we get $\sum_{i=1}^m c(x_i) \le 3m$. So $f(\eta, m) \le \eta + 3m$. Here, equality can be achieved by taking $x_1 = \eta - 3(m-1)$ and $x_2 = \cdots = x_m = 3$. (Note that $\eta \ge 3m$ implies $\eta - 3(m-1) \ge 3$, and thus $c(x_1) = 3$.) Hence $f(\eta, m) = \eta + 3m$.

    Combining the three cases above, we have
    \[
        f(\eta, m) = \begin{cases}3\eta & \text{if $\eta \le m$,}\\\lfloor 3(\eta + m)/2\rfloor & \text{if $m \le \eta \le 3m$,}\\ \eta + 3m & \text{if $\eta \ge 3m$},\end{cases}
    \]
    which matches the desired exponents in the expression for $\W{\eta}(G_t^{\oplus m})$.
\end{proof}

\section{Failures of unimodality}\label{sec:unimodality}

Throughout this section, we let $q$ be a prime power and let $\mathbb F_q$ be the finite field of cardinality $q$. Let $M$ be a simple $\mathbb F_q$-representable matroid of rank $r$ on $n$ elements.
Fixing a linear $\mathbb F_q$-representation of $M$, we identify $E(M)$ with the set of vectors in $\mathbb{F}_q^r$ corresponding to this representation.
Embed $\mathbb F_q^r$ into $\mathbb F_q^{r+1} = \mathbb F_q \oplus \mathbb F_q^r$ as the second factor. We consider the matroid $\widetilde{M}$ of rank $r+1$ consisting of the vector $a = (1, 0, \dots, 0) \in \mathbb F_q^{r+1}$ and the vectors  $(z,v)$, where $z$ runs over $\mathbb F_q$ and $v$ runs over $E(M)$.
This operation was introduced and studied by Whittle in \cite{Whittle89}, where $\widetilde{M}$ is called the \textit{$q$-lift} of $M$. 

Perhaps surprisingly, Oxley and Whittle \cite{OxleyWhittle00} showed by an example that different initial representations of $M$ can yield non-isomorphic $q$-lifts of $M$. Nevertheless, it was shown by Bonin and Qin \cite{BoninQin01} that any two $q$-lifts of $M$ have the same Tutte polynomial. They also share the same Whitney numbers of the second kind.
The following lemma of Bonin and Qin counts the number of flats of any $q$-lift $\widetilde{M}$ of $M$ in terms of the number of flats of $M$. 
It is the crucial ingredient used to turn the failures of log-concavity in \Cref{lm:gap} into failures of unimodality.
\begin{lemma}[{\cite[Lemma 4]{BoninQin01}}]\label[lemma]{lm:q-lift}
    If $M$ is a simple $\mathbb F_q$-representable matroid of rank $r$ and $\widetilde{M}$ is a $q$-lift of $M$, then for every $1\leq i\leq r$ we have
    \[
        W_i(\widetilde{M}) = q^i W_i (M) + W_{i-1}(M).
    \]
\end{lemma}

The following example illustrates the technique by which a failure of log-concavity is transformed into a failure of unimodality by an application of Whittle's $q$-lift.

\begin{example}\label[example]{ex:nonunimodal}
Consider the matroid of the generalized theta graph with four strands of lengths $1, 27, 28, 28$. This is a matroid of rank $81$ on a ground set of size $84$. We have 
\[
    (W_{77}, W_{78}, W_{79}, W_{80}) = (1648585224, 75732840, 1264437, 22275).
\]
Taking the $q$-lift for $q=59$, we obtain a matroid of rank $82$ on a ground set of size 4957 with
$$(\log_{10}W_{78}, \log_{10}W_{79}, \log_{10}W_{80}) \approx (146.00574, 145.99921, 146.01598),$$
so the Whitney numbers of this matroid are not unimodal.
\end{example}

We now explain how to apply \Cref{lm:q-lift} with even greater effect to the matroid described in \Cref{lm:asymptotics}, giving rise to severe failures of unimodality.

\begin{proof}[Proof of Theorem~\ref{thm:main}]
    Fixing an integer $m\ge1$ and letting $t$ be large, we set $M = G_t^{\oplus m}$. We will examine the asymptotic behavior of the Whitney numbers of $M$ as $t$ goes to infinity.
    Write $r = r(G_t^{\oplus m}) = m(3t-2)$. To use Whittle's construction, we choose a prime number $p$ within a factor 2 of $t^{3/2}$; such a prime exists by Bertrand's postulate.
    Since $M$ is regular, it is $\mathbb F_p$-representable, and hence we may choose a $p$-lift $\widetilde{M}$ of $M$. We write
    $\W{\eta}$ for $\W{\eta}(M)$ and $\widetilde{W}^\eta$ for $\W{\eta}(\widetilde{M})$. We will show that $\eta = m + 2s + 1$ is a peak of $\eta \mapsto \widetilde{W}^\eta$ for all $0 \leq s \leq m$. To see this,  we first use \Cref{lm:asymptotics} to deduce
    \[
        \frac{\W{m+2s-1}}{\W{m+2s}} = O(t^{-2}), \qquad 
         \frac{\W{m+2s+1}}{\W{m+2s}} = \Theta(t),
         \qquad 
         \frac{\W{m+2s+2}}{\W{m+2s}}=O(t^3).
    \]
    In fact, the first ratio is $\Theta(t^{-2})$ except when $s = 0$, in which case it is $\Theta(t^{-3})$. Similarly, the last ratio is $\Theta(t^3)$ except when $s = m$, in which case it is $\Theta(t^2)$.
    Rewriting \Cref{lm:q-lift} by corank and remembering that $\widetilde{M}$ has rank one greater than $M$, we get 
    $\widetilde{W}^{\eta} = p^{r+1-\eta}\W{\eta-1} + \W{\eta}$. Let $\ell = r - m - 2s$. Using $p = \Theta(t^{3/2})$ and $\ell = \Theta(t)$, we obtain
    \begin{align*}
        p^{-\ell}\cdot{\widetilde{W}^{m+2s}}\!/{W^{m+2s}} &= \Theta(t^{3/2}) \cdot  O(t^{-2}) + p^{-\ell}\cdot 1
        = O(t^{-1/2}),\\
       p^{-\ell}\cdot{\widetilde{W}^{m+2s+1}}\!/{W^{m+2s}} &= 
        1 + p^{-\ell}\cdot \Theta(t)
        = 1 + o(1),\\
        p^{-\ell}\cdot{\widetilde{W}^{m+2s+2}}\!/{\W{m+2s}} &= \Theta(t^{-3/2})\cdot  \Theta(t) + p^{-\ell}\cdot O(t^3)
        = \Theta(t^{-1/2}).
    \end{align*}
    For large $t$ (and $m$ fixed), this shows that $\widetilde{W}^{m+2s+1} > \max(\widetilde{W}^{m+2s}, \widetilde{W}^{m+2s+2})$.
    Since this holds for all 
    $0 \leq s \leq m$, the Whitney numbers of $\widetilde{M}$ have at least $m+1$ distinct peaks. 

    We have now seen how to construct matroids whose sequences of Whitney numbers have arbitrarily many peaks. To obtain a matroid with a \emph{prescribed} number of peaks, we start with a matroid having a potentially larger number of peaks and then reduce to the number we want by applying truncations.
    For details of this operation, see \cite[\textsection~7.3]{Oxley}. For our purposes, it will suffice to note that the Whitney numbers of the truncation of a matroid are obtained from its Whitney numbers by omitting the $W^1$ term. It is clear that the number of peaks goes down by at most one under such an omission and that we can thereby achieve any desired number of peaks.
\end{proof}

\bibstyle{amsalpha}
\bibliography{matroid}

\end{document}